\documentclass[11pt, leqno]{article}

\usepackage{amssymb}
\usepackage{amsmath}
\usepackage{amsthm}
\usepackage[pdftex]{graphicx}
\usepackage{url}

\newcommand{\mcm}[2]{\text{{\rm lcm}}(#1,#2)}
\newcommand{\mcd}[2]{\text{{\rm gcd}}(#1,#2)}
\newcommand{\zz}{\mathbb{Z}}

\newtheorem{teor}{Theorem}
\newtheorem{prop}[teor]{Proposition}
\newtheorem{lema}[teor]{Lemma}
\newtheorem{defi}[teor]{Definition}
\newtheorem{conj}[teor]{Conjecture}
\newtheorem{cor}[teor]{Corollary}
\newenvironment{demos}[1][\proofname]{\noindent\normalfont{\itshape
#1{:}}\quad\mdseries\ignorespaces}
{{$\Box$}{\vskip\belowdisplayskip}}

\title{On the use of the least common multiple to build a prime-generating recurrence}
\author{Seraf\'{\i}n Ruiz-Cabello}
\AtEndDocument{\bigskip{\footnotesize%
  \textsc{Department of Mathematics, Universidad Aut\'{o}noma de Ma\-drid, 28049 Madrid, Spain} \par  
  \textit{E-mail address}: \texttt{serafin.ruiz@uam.es} \par
}}

\begin{document}
\maketitle

\begin{abstract}
 We study a recursively defined sequence which is constructed using the least common multiple. Several authors have conjectured that every term of that sequence is $1$ or a prime. In this paper we show that this claim is connected to a strong version of Linnik's theorem, which is yet unproved. We also study a generalization that replaces the first term by any positive integer. Under this variation now some composite numbers may appear. We give a full characterization for these numbers.
\end{abstract}


\section{Introduction} \label{intro}

In 2003 Matthew Frank introduced the sequence $\{a'_n\}$ defined by
\[
 a'_n = \begin{cases} 7 & \text{for }n=1 \\ a'_{n-1} + \mcd{n}{a'_{n-1}} & \text{for }n\geq 2, \end{cases}.
\]
Computations suggested that the difference between consecutive terms, $b'_n := a'_n - a'_{n-1}$ for $n\geq 2$, was always $1$ or a prime. This result was proved by Rowland \cite{rowland}. Chamizo, Raboso, and the current author proved that the sequence $\{b'_n\}$ contained infinitely many primes \cite{crr} and several interesting variants of this sequence were studied \cite{cloitre}. In 2008 Benoit Cloitre considered
\begin{equation}\label{an}
 a_n = \begin{cases} 1 & \text{for }n=1 \\ a_{n-1} + \mcm{n}{a_{n-1}} & \text{for }n\geq 2, \end{cases} \qquad\text{and}\qquad b_n = \frac{a_n}{a_{n-1}}-1,\,\, n\geq 2.
\end{equation}

It is easy to check that every term of $\{a_n\}$ is nonzero and a multiple of the previous one. Thus the sequence $\{b_n\}$ is well defined and all terms are positive integers. Let us take a look at the first ones:

\

$\{b_n\} = \{$2, 1, 2, 5, 1, 7, 1, 1, 5, 11, 1, 13, 1, 5, 1, 17, 1, 19, 1, 1, 11, 23, 1, 5, 13, 1, 1, 29, 1, 31, 1, 11, 17, 1, 1, 37, 1, 13, 1, 41, 1, 43, 1, 1, 23, 47, 1, 1, 1, 17, 13, 53, 1, 1, 1, 1, 29, 59, 1, 61, 1, 1, 1, 13, 1, 67, 1, 23, 1, 71, 1, 73, 1, 1, 1, 1, 13, 79, 1, 1, 41, 83, 1, 1, 43, 29, 1, 89, 1, 13, 23, 1, 47, 1, 1, 97, 1, 1, 1, 101, 1, 103, 1, 1, 53, 107, 1, 109, 1, 1, 1, 113, 1, 23, 29, 1, 59, 1, 1, 1, 61, 41, 1, 1, 1, 127, 1, 43, 1, 131, 1, 1, 67, 1, 1, 137, 1, 139, 1, 47, 71, 1, 1, 29, 73, 1, 1, 149, 1, 151, 1, 1, 1, 1, 1, 157, 1, 53, 1, 1, 1, 163, 1, 1, 83, 167, 1, 13, 1, 1, 43, 173, 1, 1, 1, 59, 89, 179, 1, 181, 1, 61, 1, 1, 1, 1, 47, 1, 1, 191, 1, 193, 1, 1, 1, 197, 1, 199, 1, 67, 101, $\ldots\}$

\

At first glance we notice two details. It seems that the sequence does not contain any composite number and hence there are just ones and primes. Moreover, it looks like {\it every} prime number appears, except maybe $3$. In consequence, one may conjecture these two facts:
\begin{conj}\label{conj}
 For any $n\geq 2$, $b_n$ is either $1$ or a prime number.
\end{conj}
\begin{conj}\label{conj_3}
 The sequence $\{b_n\}$ contains every prime number other than $3$, which never appears.
\end{conj}

As far as we know, there is no proof for Conjecture \ref{conj}, although numerical evidences suggest that it probably holds. On the other hand, M. Schepke proved \cite[Thm.\ 3.10]{markus} that for any prime distinct than $3$, the term $b_p$ is equal to $p$. Hence, the remaining open question for Conjecture \ref{conj_3} is if number $3$ can ever show up in the sequence.

\

In Section \ref{main} we give a full proof for Conjecture \ref{conj_3} (direct consequence of Propositions \ref{prop_p} and \ref{prop_3}) and a sufficient condition (see Proposition \ref{suff_cond}) for Conjecture \ref{conj} that links it to a well-known theorem proved by Linnik. In a general way, Linnik's theorem asserts that there exist positive constants $c$ and $L$ such that the first prime in the arithmetic progression $a, a+d, a+2d, \ldots$ is less than $cd^L$ for any coprime integers $a$ and $d$ with $1\leq a<d$ \cite{linnik}, for some positive constants $c$ and $L$. The best known bound for $L$ is $5$ \cite{xylouris}, and it is conjectured that the theorem is still true for $L=2$ and $c=1$ \cite{heath-brown}. This stronger statement implies that for any prime $p$, the sequence $p-1, 2p-1, \ldots p^2-1$ should contain at least one prime number. If this claim turns out to be false even for a single prime, then Conjecture \ref{conj} would be false too (as a consequence of Proposition \ref{suff_cond}).

In Section \ref{general} we deal with \eqref{an} when $a_1$ is not $1$ but any positive integer. Under this variant, now it is possible to find composite odd numbers on the sequence, depending on the value of $a_1$. We shall give a necessary and sufficient condition for such numbers to appear or not on these sequences (see Theorem \ref{goodbad}). Finally, Section \ref{examples} contains some examples and calculations that support our claims and explain the behavior of the sequence \eqref{an}.


\section{Auxiliary tools and proofs of the main results}\label{main}

As we shall prove, the key tool to explain why Conjecture \ref{conj} seems to hold is the fact that for any prime $p$ and every integer $k\geq 1$, computations suggest that there are at least $k$ primes which are congruent to $-1$ modulo $p$ and less than $p^{k+1}$ (see Section \ref{examples} for further information). And actually this is more than enough. Here and subsequently, $\pi(x;q,a)$ denotes the number of primes which are less than $x$ and congruent to $a$ modulo $q$, for  integers $1\leq a\leq q$.

\begin{conj}\label{suff_conj}
 For any prime $p$ and any integer $k\geq 1$,
 \begin{equation}\label{upper_primes}
  \pi(p^{k+1};p,p-1)\geq k,
 \end{equation}
\end{conj}

For $k=1$, there are some cases on which the equality in \eqref{upper_primes} holds. Namely, $2$, $5$ and $13$ seem to be the only ones. For larger values of $p$, the quantity $\pi(p^{k+1};p,p-1)$ grows quickly as suggested by Montgomery's conjecture \cite[\S 13]{monvau}. Apparently, \eqref{upper_primes} is always a strict inequality for $k> 1$ (see Figure \ref{RR} and Table \ref{tablaa} on page \pageref{RR}). It is not possible however to claim that no counterexamples can ever be found.

\begin{prop}\label{suff_cond}
 If Conjecture \ref{suff_conj} holds, then $b_n$ can only be $1$ or the largest prime factor of $n$ for every $n\geq 2$. Hence, Conjecture \ref{suff_conj} implies Conjecture \ref{conj}.
\end{prop}

\begin{prop}\label{prop_p}
 For any prime $p\neq 3$, we have $b_p = p$.
\end{prop}

\begin{prop}\label{prop_3}
 Any positive integer $n$ satisfies $b_n\neq 3$.
\end{prop}

To prove these three Propositions, we shall employ the following Lemma:

\begin{lema}\label{lema_aux}
 Sequences $\{a_n\}$ and $\{b_n\}$ satisfy the following results:
 \begin{enumerate}
  \item $a_{n-1}$ is greater than $n$ for every $n\geq 4$.
  \item If a prime $p$ divides $a_n$ for a given $n\geq 1$, then $p\leq n+1$.
  \item Every term in $\{b_n\}$ can be written as 
  \begin{equation}\label{bn_alt}
  b_n = \frac{1}{a_{n-1}}(a_n-a_{n-1}) = \frac{\mcm{n}{a_{n-1}}}{a_{n-1}} = \frac{n}{\mcd{n}{a_{n-1}}}.
  \end{equation}
  In particular, $b_n$ is always a divisor of $n$.
 \end{enumerate}
\end{lema}

\begin{demos}
 The first point follows from the fact that every term of $\{a_n\}$ is at least twice as big as the previous one, and from $a_3>4$. For the second one, it is enough to note $a_{n} = a_{n-1} (1 + \mcm{n}{a_{n-1}}/a_{n-1}) \leq a_{n-1}(1+n)$ and to apply induction. Finally, the last one is obtained from \eqref{an} using $ab=\mcd{a}{b}\cdot\mcm{a}{b}$.
\end{demos}

\begin{proof}[Proof of Proposition \ref{prop_p}]
The case $p=2$ is straightforward. For $p\geq 5$ prime we write $b_p = p/\gcd(p,a_{p-1})$ from \eqref{bn_alt}. It is sufficient to prove that $p$ and $a_{p-1}$ are coprime. Clearly $p$ does not divide $a_{p-2}$ (as a consequence of the second point of Lemma \ref{lema_aux}); on the other hand, since $p-1$ is even, $2$ is a common divisor of it with $a_{p-1}$, because $a_n$ is even for $n\geq 3$. Hence
\[
 a_{p-1} = a_{p-2} + \mcm{p-1}{a_{p-2}} \leq a_{p-2} + \frac{p-1}{2}a_{p-2},
\]
and then $a_{p-1}/a_{p-2} \leq (p+1)/2 < p$. As $p$ does not divide the product $(a_{p-1}/a_{p-2})a_{p-2}$ and is a prime, we conclude that it is coprime with $a_{p-1}$, which completes the proof.
\end{proof}

\begin{proof}[Proof of Proposition \ref{suff_cond}]
 Conjecture \ref{conj} holds for $b_2$ and $b_3$. Let us fix $m\geq 4$ (by Lemma \ref{lema_aux}, this implies $m<a_{m-1}$), and let $p$ be the largest prime factor of $m$. We only need to show that $m/p$ is a divisor of $a_{m-1}$, since in that case $\mcm{m}{a_{m-1}}$ can only be $a_{m-1}$ or $pa_{m-1}$ and thus $b_m$ can only be $1$ or $p$ by \eqref{bn_alt}.
 
 The crucial fact is that every term in the sequence $(a_n)$ is a multiple of the preceding one (in fact, of $a_n'$ for every $n'<n$). From the first equality in \eqref{bn_alt}, it is clear that $a_n = a_{n-1}(b_n+1)$, and in consequence every $a_n$ can be expressed in terms of the sequence $(b_n)$, since by induction we have
 \begin{equation}\label{acumulador}
  a_n = \prod_{k=2}^n(1+b_k),\qquad n\geq 2.
 \end{equation}
 For a fixed integer $m$, we consider its unique prime factorization,
 \[
  m = p_1^{\alpha_1}p_2^{\alpha_2}\cdot\ldots\cdot p_k^{\alpha_k},
 \]
 with $p=p_k$ and (only if $k>1$) $p_1,\ldots,p_{k-1}<p$. For $j<k$, we have $m\geq p_j^{\alpha_j}\cdot p > p_j^{\alpha_{j}+1}$. By \eqref{upper_primes}, there exist at least $\alpha_j$ primes $q_1,q_2,\ldots,q_{\alpha_j}$ less than $p_j^{\alpha_{j+1}}$ (and hence less than $m$) which are congruent with $-1$ modulo $p_j$. For each one of them, Proposition \ref{prop_p} implies $b_{q_l}=q_l$ and therefore $a_{q_l}=a_{q_{l}-1}(1+q_l)$. Since $p_j$ divides $1+q_l$, it must divide $a_{q_l}$ at least one more time than it divides $a_{q_l-1}$. Thus $a_{m-1}$ contains the factor $p_j$ at least $\alpha_j$ times.

The case $p_k$ is very similar, except for the fact that now we cannot guarantee $m\geq p_k^{\alpha_k+1}$ but only $m\geq p_k^{\alpha_k}$. We apply the same argument as above in order to prove that $p_k$ divides $a_{m-1}$ at least $\alpha_k-1$ times and the proof is finished.
\end{proof}

For the Proof of Proposition \ref{prop_3} we also need an explicit lower bound on the number of prime numbers which are congruent to $2$ modulo $3$. Namely we use \cite{dussart}
 \begin{equation}\label{exp_bnd_2}
  \pi(x;3,2) > \frac{x}{2\log x}\qquad\text{for}\quad x\geq 151.
 \end{equation}
\begin{lema}\label{pi}
 For every integer $k\geq 0$, we have
 \begin{equation}\label{k}
  \pi(3^{k};3,2) \geq k.
 \end{equation}
\end{lema}

\begin{demos}
  It is easy to check that \eqref{k} holds for $0\leq k\leq 4$. For larger values of $k$ we use \eqref{exp_bnd_2} and the fact that $3^xx^{-2}$ is an increasing function from $x=5$ onwards:
 \[
  \pi(3^{k};3,2) > \frac{3^k}{(\log 9)k} \geq \frac{3^5}{(\log 9)5^2}\,k > k.
 \]
\end{demos}

Note that Lemma \ref{pi} actually is a stronger version of \eqref{upper_primes} for $p=3$. This is the only prime such that $p-1$ is also a prime, and that is the reason why $3$ does not appear in $\{b_n\}$.

\

\begin{demos}[Proof of Proposition \ref{prop_3}]
 By \eqref{bn_alt}, it is sufficient to show that $3^k$ divides $a_{3^k}$ for any positive integer $k$. Lemma \ref{pi} implies that there exist primes $p_1,p_2,\ldots p_k$ less than $3^k$ which are congruent to $-1$ modulo $3$. By Lemma \ref{prop_p}, $b_{p_j} = p_j$, and hence $a_{p_j} = a_{p_{j-1}}(p_j + 1)$ is a multiple of $3$. The proof is concluded using that every $a_m$ is a multiple of $a_{m-1}$ and applying induction.
\end{demos}


\section{Generalization} \label{general}

Now a natural question arises: What happens if the initial condition in \eqref{an} is not $1$ but any possitive integer? For a fixed $s\geq 1$ we define
\begin{equation}\label{ans}
 a_n^s = \begin{cases} s & \text{for }n=1 \\ a_{n-1}^s + \mcm{n}{a_{n-1}^s} & \text{for }n\geq 2, \end{cases} \qquad\text{and}\qquad b_n^s = \frac{a_n^s}{a_{n-1}^s}-1,\,\, n\geq 2.
\end{equation}

The relation between differents sequences of this family is surprising (see Table \ref{tabla}). The first thing we notice is that they all seem to be very similar, with little variations that may depend on common divisors of given values of $n$ and $s$. Also, it is clear that Conjecture \ref{conj} cannot be generalized for every $s$. We can see on Table \ref{tabla} that for instance the sequence $\{b_n^{10}\}$ contains a composite number, a $9$. For suitable values of $s$ (that, as we shall discuss later, strongly depend on the composite number $m$ we want to find), many more counterexamples are found. For instance,
\[
 b_{21}^{24310} = 21,\quad b_{25}^{19} = 25,\quad b_{27}^{43010} = 27,\quad b_{35}^{7163} = 35.
\]
See Tables \ref{tablan} and \ref{tablad} from page \pageref{tablan} onwards for more examples. There are also many numbers that never appear. It is not hard to prove (see Proposition \ref{odd_s}) that no even number greater than $2$ can be found in any sequence. A little more is required to show that, for instance, $b_n^s$ is not $15$ for any pair $(n,s)$. According to this, from now on we shall classify composite numbers into two different groups:

\begin{defi}\label{deff}
 Given a composite integer $m\geq 3$, we call it a {\it present number} if there exist $n$ and $s$ such that $b_n^s = m$. Otherwise, we call it an {\it absent number}.
\end{defi}

\begin{table}
\[
 \begin{array}{|l||l|l|l|l|l|l|l|l|l|l|l|l|l|l|l|l|l|}
 \hline
  a_1^s & b_2^s & b_3^s & b_4^s & b_5^s & b_6^s & b_7^s & b_8^s & b_9^s & b_{10}^s & b_{11}^s & b_{12}^s & b_{13}^s & b_{14}^s & b_{15}^s & b_{16}^s & b_{17}^s \\
  \hline
  1 & 2 & 1 & 2 & 5 & 1 & 7 & 1 & 1 & 5 & 11 & 1 & 13 & 1 & 5 & 1 & 17 \\
  \hline
  2 & 1 & 3 & 1 & 5 & 1 & 7 & 1 & 3 & 5 & 11 & 1 & 13 & 1 & 5 & 1 & 17 \\
  \hline
  3 & 2 & 1 & 2 & 5 & 1 & 7 & 1 & 1 & 5 & 11 & 1 & 13 & 1 & 5 & 1 & 17 \\
  \hline
  4 & 1 & 3 & 1 & 5 & 1 & 7 & 1 & 3 & 5 & 11 & 1 & 13 & 1 & 5 & 1 & 17 \\
  \hline
  5 & 2 & 1 & 2 & 1 & 1 & 7 & 1 & 1 & 1 & 11 & 1 & 13 & 1 & 1 & 1 & 17 \\
  \hline
  6 & 1 & 1 & 1 & 5 & 1 & 7 & 1 & 1 & 5 & 11 & 1 & 13 & 1 & 5 & 1 & 17 \\
  \hline
  7 & 2 & 1 & 2 & 5 & 1 & 1 & 1 & 1 & 5 & 11 & 1 & 13 & 1 & 5 & 1 & 17 \\
  \hline
  8 & 1 & 3 & 1 & 5 & 1 & 7 & 1 & 3 & 5 & 11 & 1 & 13 & 1 & 5 & 1 & 17 \\
  \hline
  9 & 2 & 1 & 2 & 5 & 1 & 7 & 1 & 1 & 5 & 11 & 1 & 13 & 1 & 5 & 1 & 17 \\
  \hline
  10 & 1 & 3 & 1 & 1 & 3 & 7 & 1 & 9 & 1 & 11 & 1 & 13 & 1 & 1 & 1 & 17 \\
  \hline
  11 & 2 & 1 & 2 & 5 & 1 & 7 & 1 & 1 & 5 & 1 & 1 & 13 & 1 & 5 & 1 & 17 \\
  \hline
 \end{array}
\]
\caption{First values of $\{b_n^s\}$ for $1\leq s\leq 11$.}
\label{tabla}
\end{table}

So, besides the first open question of this chapter (can Conjectures \ref{conj} and \ref{conj_3} be generalized for a fixed $s$?), another one -which turns out to be much more interesting- arises: can we classify composite numbers into absent ones and present ones? And the answer is positive; Theorem \ref{goodbad} (among with Propositions \ref{prop_ps} and \ref{odd_s}) gives a full characterization.

But first let us say what we can now about the presence of composite numbers on the sequence $\{b_n^s\}_{n\geq 1}$ for a given $s$. Depending on the factors of $s$, sometimes it is possible to establish an algorithm in order to show that $s$ is allowing any present number to appear. The simplest examples are $s=19$ and $s=103$ (or many of their multiples), that lead to $b_{5^2}^{19}=5^2$ and $b_{13^2}^{103}=13^2$. As we discussed before, this is related to the fact that $5$ and $13$ are probably the only odd primes such that the equality in \eqref{upper_primes} holds for $k=1$. If we get a negative result for that particular $s$, then Conjecture \ref{conj} can be generalized to $\{b_n^s\}$, by using a slightly modified version of \eqref{upper_primes} (that changes for every $s$). It seems difficult, however, to establish a general constructive method for any large $s$. Regarding Proposition \ref{prop_p}, it can be extended for any $s$ (see Proposition \ref{prop_ps}), but not Proposition \ref{prop_3} (see the remark after Proposition \ref{odd_s}). First, we need to introduce this extended version of Lemma \ref{lema_aux}:

\begin{lema}\label{lema_auxs}
 Fix $s\geq 1$. Sequences $\{a_n^s\}$ and $\{b_n^s\}$ satisfy the following results:
 \begin{enumerate}
  \item $a_{n-1}^s$ is greater than $n$ for every $n\geq 4$ and $s\geq 1$.
  \item If a prime $p$ divides $a_n^s$ for a given $n\geq 1$ and $p\nmid s$, then $p\leq n+1$.
  \item $b_n^s$ can be written as 
  \begin{equation}\label{bn_alts}
  b_n^s = \frac{\mcm{n}{a_{n-1}^s}}{a_{n-1}^s} = \frac{n}{\mcd{n}{a_{n-1}^s}}.
  \end{equation}
  In particular, $b_n^s$ is always a divisor of $n$.
  \item A given $a_n^s$ for $n\geq 2$ can be expressed in terms of $b_2^s,b_3^s,\ldots,b_n^s$. Namely,
  \begin{equation}\label{abb}
   a_n^s = \prod_{k=2}^n(1+b_k^s),\qquad n\geq 2.
  \end{equation}
 \end{enumerate}
\end{lema}

\begin{proof}
 Taking into account that
 \begin{equation}\label{4ss}
  (a_2^s,a_3^s,a_4^s) = \begin{cases} (2s,4s,8s) & \text{if    }s\equiv 0\pmod{6} \\ (2s,8s,16s) & \text{if    }s\equiv 2,4\pmod{6} \\ (3s,6s,18s) & \text{if    }s\equiv 1\pmod{2}\end{cases},
 \end{equation}
 and that if $p$ divides $a_n$, then either $p\leq n+1$ or $p$ divides $s$, the proof of Lemma \ref{lema_aux} is still suitable for $\{a_n^s\}$ and $\{b_n^s\}$. Note that \eqref{abb} is obtained as \eqref{acumulador}.
\end{proof}

\begin{prop}\label{prop_ps}
 Fix an integer $s\geq 1$ and a prime $p\neq 3$. We have $b_4^s=b_2^s$ and
 \[
  b_p^s = \begin{cases} p & \text{if    }p\nmid s \\ 1 & \text{if    }p\mid s\end{cases}, \qquad\quad b_3^s = \begin{cases} 3 & \text{if    }s\equiv 2,4\pmod{6} \\ 1 & \text{if    }s\not\equiv 2,4\pmod{6}.\end{cases}
 \]
\end{prop}

\begin{prop}\label{odd_s}
 Given any initial condition $s$, $b_n^s$ is always odd for $n\geq 5$.
\end{prop}

Note that Proposition \ref{prop_ps} is a generalization of Proposition \ref{prop_p}. However, it is not always possible to do the same with Proposition \ref{prop_3}, because for a large prime $s$ we cannot prove $b_n^s\neq s$ for every $n$ as we did for $\{b_n^1\}$ on Lemma \ref{pi} (this would need an explicit lower bound for $\pi(x;s,s-1)$, which do not always exists for large values of $s$). Hence a full generalization of Conjecture \ref{conj_3} holds whenever $s$ is not a large prime. Note also that combining these two Propositions it is clear that, if $m\geq 4$ is an even integer, then $b_n^s \neq m$ for every possible values of $n$ and $s$.

\

\begin{demos}[Proof of Proposition \ref{prop_ps}]
By \eqref{4ss}, the Proposition is straightforward for $n=2,3,4$. For a prime $p\geq 5$, if $p$ divides $s$, then $a_{p-1}$ is a multiple of $p$ and by \eqref{bn_alts}, $b_p = p/\mcd{a_{p-1}}{p} = 1$. And if $p\nmid s$, we can proceed as in the proof of Proposition \ref{prop_p}, getting that $p$ and $a_{p-1}^s$ are coprime and hence $b_p^s = p$.
\end{demos}

\begin{demos}[Proof of Proposition \ref{odd_s}]
 We employ an explicit lower bound on the number of primes which are less than a fixed $x$ \cite{rosser}, \begin{equation}\label{exp_bnd_1}
  \pi(x;2,1) > \frac{x}{\log x+2}-1\qquad\text{for}\quad x\geq 55.
 \end{equation}
 Using this inequality and proceeding as in the proof of Lemma \ref{pi}, it is easy to check that $\pi(2^k,2,1)\geq k$ for $k\geq 3$. Combining this with \eqref{4ss}, we have $8|a_7^s$ and by induction, $2^k|a_{2^k-1}^s$ for $k\geq 3$.
\end{demos}

For the next results, we shall use the following notation: for any two integers $q$ and $m$, we will write $q\vartriangleleft m$ if every prime divisor of $q$ is also a prime divisor of $m$ and $q<m$. For instance, $9\vartriangleleft 15$, $27\ntriangleleft 15$ or $6\ntriangleleft 9$.

\begin{teor}\label{goodbad}
 Let $m\geq 3$ be an odd composite integer. It there exist $p$ and $q$ such that $p$ is a prime divisor of $m$, $q\vartriangleleft m$ and $q\equiv -1\pmod{p}$, then $m$ is an absent number. Otherwise there exists an integer $s$ such that $b_m^s = m$ and in consequence $m$ is a present number.
\end{teor}

This result gives us a method in order to know whether an odd composite number is absent or not. It is clear from the Definition that every power of an odd prime is present. Actually, they are easy to find as explicit counterexamples. For instance,
\[
 b_{7^2}^{533} = 7^2,\quad b_{11^2}^{4687} = 11^2, \quad b_{13^2}^{103} = 13^2, \quad b_{5^3}^{s_1} = 5^3,\quad\text{for}\quad s_1=19\cdot 29\cdot 59\cdot 79\cdot 89\cdot 109.
\]
For numbers with more than prime divisors, the condition of Theorem \ref{goodbad} is tricky and there are several possibilities, even with just two different primes. For instance, $21$ is present, since $7\not\equiv -1\pmod{3}$. But $63$ is absent, because $3^3 \equiv -1\pmod{7}$. The smallest absent number is $15$ and, in the same way, every number of the form $3p$ for a prime $p\equiv 5\pmod{6}$ is an absent one. The more different divisors has $m$, the smallest are the chances for it to be present. The following result is straightforward:

\begin{cor}\label{sonofbad}
 Every multiple of an absent number is also absent. Also, every divisor of a present number is present.
\end{cor}

The proof of the theorem is split in two parts. For the one concerning the absent numbers, we use Lemma \ref{goodlemma} and a recursive step (Proposition \ref{goodprop}).

\begin{lema}\label{goodlemma}
 If $m$ is an absent number for a given pair $(p,q)$, then we can always write $m=pqm_0$, and $p$ and $q$ are necessarily odd with $p<q$ and $\gcd(p,q)=1$.
\end{lema}

\begin{prop}\label{goodprop}
 Let $m=pqm_0$ be an absent number for a pair $(p,q)$. For every integers $c,d\geq 1$ and $t:=p^cq^dm_0$, we claim that
\begin{enumerate}
 \item If $b_t^s = 1$, then $b_{pqt}^s \in\{1,p,q\}$.
 \item If $b_t^s = p$, then $b_{qt}^s  \in\{1,p,q\}$.
 \item If $b_t^s = q$, then $b_{pt}^s  \in\{1,q\}$.
 \item If $b_t^s = j\in\{1,p,q\}$, then $b_k^s\neq m$ for $t\leq k<(pq/j)t$.
\end{enumerate}
\end{prop}

\begin{demos}
\begin{enumerate}
 \item If $b_t^s = 1$, then it means $p^cq^d|a_{t-1}^s$ by \eqref{bn_alts}. We distinguish two cases now. If $q^{d+1}|a_{qt-1}^s$, then $p^cq^{d+1}|a_{pqt-1}^s$, and in consequence $b_{pqt}^s$ can only be $1$ or $p$. If not, then we have $p^cq^d|a_{qt-1}^s$ and $q^{d+1}\nmid a_{qt-1}^s$, and hence $a_{qt}^s = (q+1)a_{qt-1}^s$, which implies $p^{c+1}q^d|a_{pqt-1}^s$. Thus $b_{pqt}^s$ can only be $1$ or $q$.
 \item If $b_t^s = p$, again by \eqref{bn_alts} we obtain $p^{c-1}q^d|a_{t-1}^s$ and $p^c\nmid a_{t-1}^s$. Since $p<q$, we have $t<p^{c-1}q^{d+1}=tq/p$. We distinguish two cases once more. If $q^{d+1}$ divides $a_{tq/p-1}^s$, then it divides $a_{qt-1}^s$ too. As a consequence, $b_{qt}^s$ can only be $1$ or $p$. If $q^{d+1}$ does not divide $a_{tq/p-1}^s$ (and recalling that $q^d|a_{tq/p-1}^s$), $b_{tq/p}^s$ is $q$ and $a_{tq/p}^s=(q+1)a_{tq/p-1}^s$. This implies $p|a_{tq/p}^s$ and $p|a_{pqt-1}^s$, from which $b_{pqt}^s$ can only be $1$ or $q$.
 \item The last case is $b_t^s = q$. As before, from it we obtain $p^cq^{d-1}|a_{t-1}^s$ and $q^d\nmid a_{t-1}^s$. But necessarily $a_t^s = (q+1)a_{t-1}^s$, and since $q\equiv -1\pmod{p}$, we conclude $p^{c+1}|a_{t}^s$ and $b_{pt}^s$ can only be $1$ or $q$.
 \item Finally, let us suppose $b_t^s = j\in\{1,p,q\}$. As every $b_n^s$ is a divisor of $n$, if we want to find $k$ such that $k\geq t$ and $b_k^s = m$, the least possible candidate is $t\cdot (pq/j)$. This proves the fourth and last point.
\end{enumerate}
\end{demos}

\begin{demos}[Proof of Theorem \ref{goodbad}]
We distinguish two cases. First, let us assume that for a particular $m$ there exist $p$ and $q$ as in the statement, that we can write in the form $pqm_0$ after Lemma \ref{goodlemma}. Then it is a matter of combining the four properties of Proposition \ref{goodprop} and apply induction. For any $s\geq 1$, clearly $b_{pq}$ can only be $1$, $p$ or $q$ (if $q$ divides $a_{q-1}^s$, then $q\nmid b_{pq}^s$; and if not, then $a_q^s = (q+1)a_{q-1}^s$ and then $p|a_{pq-1}^s$ and $p\nmid b_{pq}^s$). We are now in one of the three possibilities, for $c=d=1$, and fall again in one of them, for $(c,d)\in\{(1,2),(2,2),(2,1)\}$. Iterating this process, we obtain an infinite sequence $\{(c_j,d_j)\}_{j\geq 1}$, with $c_j \leq c_{j+1}\leq c_{j}+1$ and $d_j \leq d_{j+1}\leq d_{j}+1$. And for every $n\geq 1$, there is a single $j$ such that $p^{c_j}q^{d_j}m_0\leq n < p^{c_{j+1}}q^{d_{j+1}}m_0$. By the fourth point, $b_n^s \neq m$.

\

The last case is when no pair $(p,q)$ satisfies the conditions of Theorem \ref{goodbad} for a fixed $m$. Then we take
\[
 S = \{1\leq l\leq m:\mcd{l}{m}=1\},\qquad\text{and}\quad s := \prod_{l\in S} l.
\]
Now let $p_1$, $\ldots$, $p_k$ be the prime divisors of number $m$. We employ the following sets:
 \[
  P_j = \{1\leq n\leq m:n\equiv -1\pmod{p_j}\}, \qquad P = \bigcup_{j=1}^k P_j,
 \]
 \[
  \widetilde{P} = \{1\leq kp\leq m:k\vartriangleleft m,p\in P\}, \qquad\text{and}\quad \widehat{P} := \{1\leq kp\leq m:k\in\zz,p\in P\}.
 \]

Note that always $1\not\in P\subset \widetilde{P}\subset\widehat{P}$. Since clearly $m$ and $s$ are coprime, it is enough to prove $b_n^s \not\in P$ for $2\leq n<m$ because in that case, by \eqref{abb}, $m$ and $a_{m-1}^s$ are coprime too and then $b_m^s=m$ by \eqref{bn_alts}. We shall prove this by complete induction. For the first step we recall that $s$ is an odd number, and hence $b_2^s = 1\not\in P$ by Proposition \ref{prop_ps}. Now, for a given $2\leq n<m$, let us suppose that $b_j^s \not\in P$ for $j<n$, which implies that $n$ and $a_{n-1}^s$ are coprime. Using once more that $b_n^s$ is a divisor of $n$, we can restrict ourselves to the case $n\in \widehat{P}$; in other words, we can assume $n$ to be a multiple of an element of $P$. Now there are two possibilities. If $n$ is in $S$, then by \eqref{bn_alts}, $b_n^s$ is $1$, which is not contained in $\widehat{P}$. Otherwise, we have $n\in \widehat{P}\setminus S$. Let us call $n_0$ the greatest multiple of $n$ such that $n_0$ and $m$ are coprime and write $n=n_0n_1$. Since $n\not\in L$, this clearly implies $n_0\in L$ and $1<n_1\ntriangleleft m$. Then
\[
 b_n^s = \frac{n_0n_1}{\mcd{n_0n_1}{a_{n-1}^s}} = n_1.
\]
Finally, every possible prime divisor of $n_1$ is contained in $\{p_1,\ldots,p_k\}$. But then there cannot exist $j$ such that $n_1\equiv -1\pmod{p_j}$, since $n_1$ is a divisor of a present number, and hence a present number too by Corollary \ref{sonofbad}.
\end{demos}

\begin{table}
\[
 \begin{array}{|c|c|c|}
 \hline
 m & p^k & s \\
 \hline
 9 & 3^2 & 2\cdot 5 \\
 \hline
 25 & 5^2 & 19 \\
 \hline
 27 & 3^3 & 2\cdot 5\cdot 11\cdot 17\cdot 23 \\
 \hline
 49 & 7^2 & 13\cdot 41 \\
 \hline
 81 & 3^4 & 2\cdot 5\cdot 11\cdot 17\cdot 23\cdot 29\cdot 41\cdot 47\cdot 53\cdot 59\cdot 71 \\
 \hline
 121 & 11^2 & 43\cdot 109 \\
 \hline
 125 & 5^3 & 19\cdot 29\cdot 59\cdot 79\cdot 89\cdot 109 \\
 \hline
 169 & 13^2 & 103 \\
 \hline
 243 & 3^4 & 2\cdot 5\cdot 11\cdot 17\cdot 23\cdot 29\cdot 41\cdot 47\cdot 53\cdot 59\cdot 71\cdot 83\cdot 89\cdot 101\cdot 107\cdot 113\cdot 131\cdot  \\
 & & \cdot 137\cdot 149\cdot 167\cdot 173\cdot 179\cdot 191\cdot 197\cdot 227\cdot 233\cdot 239\\
 \hline
 289 & 17^2 & 67\cdot 101\cdot 271 \\
 \hline
 343 & 7^3 & 13\cdot 41\cdot 83\cdot 97\cdot 139\cdot 167\cdot 181\cdot 223\cdot 251\cdot 293\cdot 307\\
 \hline
 361 & 19^2 & 37\cdot 113\cdot 151\cdot 227\\
 \hline
 529 & 23^2 & 137\cdot 229\cdot 367\\
 \hline
 625 & 5^4 & 19\cdot 29\cdot 59\cdot 79\cdot 89\cdot 109\cdot 139\cdot 149\cdot 179\cdot 199\cdot 229\cdot 239\cdot 269\cdot 349\cdot 359\cdot \\
 & & 379\cdot 389\cdot 409\cdot 419\cdot 439\cdot 449\cdot 479\cdot 499\cdot 509\cdot 569\cdot 599\cdot 619 \\
 \hline
 729 & 3^5 & 2\cdot 5\cdot 11\cdot 17\cdot 23\cdot 29\cdot 41\cdot 47\cdot 53\cdot 59\cdot 71\cdot 83\cdot 89\cdot 101\cdot 107\cdot 113\cdot 131\cdot  \\
 & & \cdot 137\cdot 149\cdot 167\cdot 173\cdot 179\cdot 191\cdot 197\cdot 227\cdot 233\cdot 239 \cdot 251\cdot 257\cdot 263\cdot \\
 & & \cdot 269\cdot 281\cdot 293\cdot 311\cdot 317\cdot 347\cdot 353\cdot 359\cdot 383\cdot 389\cdot 401\cdot 419\cdot 431\cdot  \\
 & & \cdot 443 \cdot 449\cdot 461\cdot 467\cdot 479\cdot 491\cdot 503\cdot 509\cdot 521\cdot 557\cdot 563\cdot 569 \cdot 587 \cdot\\
 & & \cdot 593\cdot 599\cdot 617\cdot 641\cdot 647\cdot 653\cdot 659\cdot 677\cdot 683\cdot 701\cdot 719\\
 \hline
 841 & 29^2 & 173\cdot 347\cdot 463\cdot 521\cdot 811\\
 \hline
 961 & 31^2 & 61\cdot 433\cdot 557\cdot 619\cdot 743\cdot 929\\
 \hline
 1331 & 11^3 & 43\cdot 109\cdot 131\cdot 197\cdot 241\cdot 263\cdot 307\cdot 373\cdot 439\cdot 461\cdot 571\cdot 593\cdot 659\cdot 769\cdot \\
 & & \cdot 857\cdot 967\cdot 1033\cdot 1187\cdot 1231\cdot 1297\cdot 1319\\
 \hline
 1369 & 37^2 & 73\cdot 443\cdot 739\cdot 887\cdot 1109\\
 \hline
 \vdots & \vdots & \vdots \\
 \hline
 \end{array}
\]
\caption{Non trivial powers of odd primes, and the least $s$ such that $b_m^s = m$.}
\label{tablan}
\end{table}

\begin{table}
\[
 \begin{array}{|c|c|c|c|c|}
 \hline
 m & \prod p_j^{\alpha_j} & \text{G/B} & p,q & s \\
  \hline
 15 & 3\cdot 5 & \text{Good} & 5\equiv -1\,(3) & \\
  \hline
 21 & 3\cdot 7 & \text{Bad} & & 2\cdot 5\cdot 11\cdot 13\cdot 17 \\
  \hline
 33 & 3\cdot 11 & \text{Good} & 11\equiv -1\,(3) & \\
  \hline
 35 & 5\cdot 7 & \text{Bad} & & 13\cdot 19\cdot 29 \\
  \hline
 39 & 3\cdot 13 & \text{Bad} & & 2\cdot 5\cdot 11\cdot 17\cdot 23\cdot 29\\
  \hline
 45 & 3^2\cdot 5 & \text{Good} & 5\equiv -1\,(3) & \\
  & & & 3^2\equiv -1\,(5) & \\
  \hline
 51 & 3\cdot 17 & \text{Good} & 17\equiv -1\,(3) & \\
  \hline
 55 & 5\cdot 11 & \text{Bad} & & 19\cdot 29\cdot 43 \\
  \hline
 57 & 3\cdot 19 & \text{Bad} & & 2\cdot 5\cdot 11\cdot 17\cdot 23\cdot 29\cdot 37 \cdot 41\cdot 47\cdot 53 \\
  \hline
 63 & 3^2\cdot 7 & \text{Good} & 3^3\equiv -1\,(7) & \\
  \hline
 65 & 5\cdot 13 & \text{Good} & 5^2\equiv -1\,(13) & \\
  \hline
 69 & 3\cdot 23 & \text{Good} & 23\equiv -1\,(3) & \\
  \hline
 75 & 3\cdot 5^2 & \text{Good} & 5\equiv -1\,(3) & \\
  & & & 3^2\equiv -1\,(5) & \\
  \hline
 77 & 7\cdot 11 & \text{Bad} & & 13\cdot 41\cdot 43 \\
  \hline
 85 & 5\cdot 17 & \text{Bad} & & 19\cdot 29\cdot 59\cdot 67\cdot 79 \\
  \hline
 87 & 3\cdot 29 & \text{Good} & 29\equiv -1\,(3) & \\
  \hline
 91 & 7\cdot 13 & \text{Good} & 13\equiv -1\,(7) & \\
  \hline
 93 & 3\cdot 31 & \text{Bad} & & 2\cdot 5\cdot 11\cdot 17\cdot 23\cdot 29 \cdot 41\cdot 47\cdot 53\cdot \\
  & & & &\cdot 59\cdot 61\cdot 71\cdot 83\cdot 89 \\
  \hline
 95 & 5\cdot 19 & \text{Good} & 19\equiv -1\,(5) & \\
  \hline
 99 & 3^2\cdot 11 & \text{Good} & 11\equiv -1\,(3) & \\
  \hline
 105 & 3\cdot 5\cdot 7 & \text{Good} & 3^2\equiv -1\,(5) & \\
  & & & 5\cdot 7\equiv -1\,(3) & \\
  & & & 7^2\equiv -1\,(5) & \\
  \hline
 \vdots & \vdots & \vdots & \vdots & \vdots \\ 
  \hline
 \end{array}
\]
\caption{First odd integers with at least two different prime divisors. If they are present, the least $s$ such that $b_m^s=s$. If they are absent, all the counterexample pairs $(p,q)$.}
\label{tablad}
\end{table}


\section{Examples and tables}\label{examples}

As we have just seen, we know for sure that there are infinitely many values of $s$ such that $\{a_n^s\}$ contains composite numbers. However, very strong divisibility conditions are needed in order to find such numbers, that we called {\it present} (many composite numbers can never show up, as we stated on Theorem \ref{goodbad}). So if we choose a random value for $s$, the odds say that probably that sequence is clean. Let us begin by taking a look at some examples of absent and present integers. Every power of a prime is present, the first ones are in Table \ref{tablan}. Note that, for any fixed present number $m$, the value $s$ that we get in the proof of the theorem is not necessarily the smallest possible. In the case $m=p^k$, actually it is enough to take the product of the primes contained in $P$, and that is what we do in our examples.

If $m$ has more than one prime divisor, then $m$ can be absent or present, as Definition \ref{deff} states. See Table \ref{tablad}. The first absent integers are $15$, $33$, $45$, $51$, $63$, $65$, $69$, $75$, $87$ and $91$. We stated that we can never find these integers in any generalized sequence. The rest of them are present, and by the proof of Theorem \ref{goodbad}, we can find suitable values of $s$ such that $b_m^s=m$. Again, it is not necessary to take so many factors to get $s$, but taking just the primes in $P$ is not enough. For instance, taking $s=3\cdot 13\cdot 19$, we need $43|s$ to get $b_m^s=m$, although $43$ is not congruent to $-1$ modulo $3$, $13$ or $19$ (the reason for this is the fact that $3\cdot 43$ is congruent to $-1$ modulo $13$).

For larger values of $m$, usually more numbers are expected to be absent, as having more different factors creates more chances. However, picking products of two different large primes, we can always find arbitrarily large present numbers. It is not easy to describe all the pairs $(n,s)$ such that a present number $m$ satisfies $b_n^s = m$; if $n$ is greater than $m$, there are endless possibilities adding and removing factors from $s$.

On the other hand, given a specific value of $s$, is there a way to check if the sequence $\{b_n^s\}$ contains present numbers? Again, there is not a clean answer. As we discussed before, the number of times that a prime $p$ divides $a_n^s$ depends not only on the primes $q$ less than $n$ and such that $q\equiv -1\pmod{p}$, but on the factors $b_j^s$ with $j\leq n$ and $b_j^s\equiv -1\pmod{p}$ (and also on the factors of $s$ itself). If $s$ is not a huge number and we can manually check $O(s)$ steps in the sequence $b_n^s$, then are able to say if there are present numbers, or if we can establish the original conjecture for this particular value of $s$.

\

Finally, let us go back to the original sequence, \eqref{an}, and say a few words about Conjecture \ref{conj}. As we can see on Table \ref{tablao}, $\{a_n\}$ grows very fast. Also, by \eqref{acumulador}, on every step it is multiplied by $(1+b_n)$ and thus it plays the role of an accumulator of $\{b_n\}$. According to the proof of Proposition \ref{suff_cond}, what $b_n$ needs to be $1$ or prime is every prime factor of $n$ to appear on $a_n$ at least the same number of times, with the single possible exception of one of them. 

\begin{table}
\[
 \begin{array}{|c|c|c|c|c|}
 \hline
 n & a_{n-1} & \mcm{n}{a_{n-1}} & b_n & a_n/a_{n-1} \\
  \hline
 2 & 1 & 2a_{1} & 2 & 3 \\
  \hline
 3 & 3 & a_{2} & 1 & 2 \\
  \hline
 4 & 2\cdot3 & 2a_{3} & 2 & 3 \\
  \hline
 5 & 2\cdot3^2 & 5a_{4} & 5 & 2\cdot 3 \\
  \hline
 6 & 2^{2}\cdot3^3 & a_{5} & 1 & 2 \\
  \hline
 7 & 2^{3}\cdot3^3 & 7a_{6} & 7 & 2^3 \\
  \hline
 8 & 2^{6}\cdot3^3 & a_{7} & 1 & 2 \\
  \hline
 9 & 2^{7}\cdot3^3 & a_{8} & 1 & 2 \\
  \hline
 10 & 2^{8}\cdot3^3 & 5a_{9} & 5 & 2\cdot 3 \\
  \hline
 11 & 2^{9}\cdot3^4 & 11a_{10} & 11 & 2^2\cdot 3 \\
  \hline 
 12 & 2^{11}\cdot3^5 & a_{11} & 1 & 2 \\
  \hline
 13 & 2^{12}\cdot3^5\cdot7 & 13a_{12} & 13 & 2\cdot7 \\
  \hline
 14 & 2^{13}\cdot3^5\cdot7 & a_{13} & 1 & 2 \\
  \hline
 15 & 2^{14}\cdot3^5\cdot7 & 5a_{14} & 5 & 2\cdot3 \\
  \hline
 16 & 2^{15}\cdot3^6\cdot7 & a_{15} & 1 & 2 \\
  \hline
  \vdots & \vdots & \vdots & \vdots & \vdots \\
  \hline
 \end{array}
\]
\caption{A deeper explanation of how $a_n$ and $b_n$ do behave.}
\label{tablao}
\end{table}

\begin{table}
\[
 \begin{array}{|c|c|c|c|}
 \hline
 n & a_{n-1} & b_n & a_n/a_{n-1} \\
 \hline
 159 & 2^{212} \cdot 3^{54} \cdot 5^{14} \cdot 7^{14} \cdot 11^{5} \cdot {13} \cdot 17^{3} \cdot 19^{3} \cdot 23 \cdot 31^2 \cdot 37^2 \cdot 79 & 53 & 2\cdot 3^3 \\
  \hline
   160 & 2^{213} \cdot 3^{57} \cdot 5^{14} \cdot 7^{14} \cdot 11^{5} \cdot {13} \cdot 17^{3} \cdot 19^{3} \cdot 23 \cdot 31^2 \cdot 37^2 \cdot 79 & 1 & 2 \\
  \hline
   161 & 2^{214} \cdot 3^{57} \cdot 5^{14} \cdot 7^{14} \cdot 11^{5} \cdot {13} \cdot 17^{3} \cdot 19^{3} \cdot 23 \cdot 31^2 \cdot 37^2 \cdot 79 & 1 & 2 \\
  \hline
   162 & 2^{215} \cdot 3^{57} \cdot 5^{14} \cdot 7^{14} \cdot 11^{5} \cdot {13} \cdot 17^{3} \cdot 19^{3} \cdot 23 \cdot 31^2 \cdot 37^2 \cdot 79 & 1 & 2 \\
  \hline
   163 & 2^{216} \cdot 3^{58} \cdot 5^{14} \cdot 7^{14} \cdot 11^{5} \cdot {13} \cdot 17^{3} \cdot 19^{3} \cdot 23 \cdot 31^2 \cdot 37^2 \cdot 79 & 163 & 2^2\cdot 41 \\
  \hline
  164 & 2^{218} \cdot 3^{58} \cdot 5^{14} \cdot 7^{14} \cdot 11^{5} \cdot {13} \cdot 17^{3} \cdot 19^{3} \cdot 23 \cdot 31^2 \cdot 37^2 \cdot 41 \cdot 79 & 1 & 2 \\
  \hline
  \vdots & \vdots & \vdots & \vdots \\
  \hline
 \end{array}
\]
\caption{Values of $a_{n-1}$ and $b_n$ for $159\leq n\leq 164$.}
\label{tablak}
\end{table}

\begin{figure}
\includegraphics[scale=0.4]{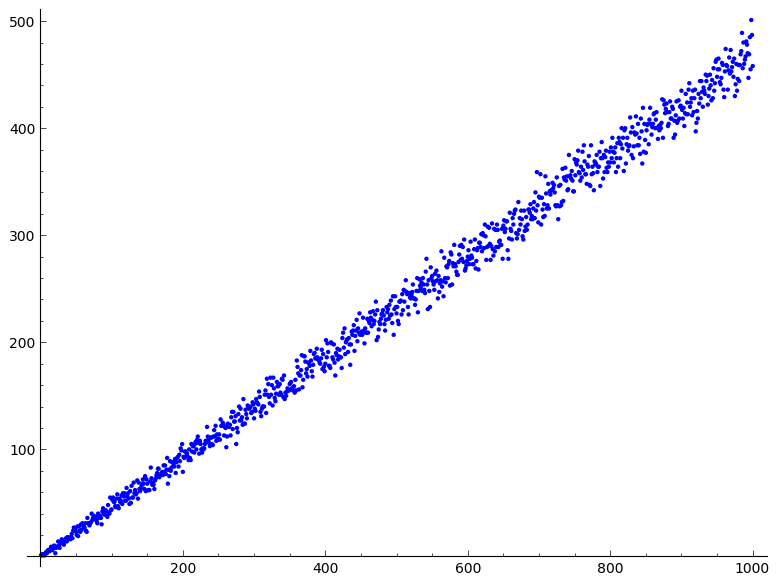}
\centering
\caption{The first $1000$ primes $p_n$ vs. $\pi(p_n^2;p_n,p_n-1)$, which is never $0$ in this range.}
\label{RR}
\end{figure}

\begin{table}
\[
 \begin{array}{|l||l|l|l|l|l|}
 \hline
  p\backslash k & 2 & 3 & 4 & 5 & 6 \\
  \hline
  \hline
   2 & 1 & 3 & 5 & 10 & 17 \\
  \hline
   3 & 2 & 5 & 11 & 27 & 67 \\
  \hline
   5 & 1 & 6 & 27 & 110 & 450 \\
  \hline
   7 & 2 & 11 & 62 & 327 & 1849 \\
  \hline
   11 & 2 & 21 & 171 & 1487 & 13295 \\
  \hline
   13 & 1 & 27 & 252 & 2603 & 28150 \\
  \hline
   17 & 3 & 41 & 502 & 6782 & 94708 \\
  \hline
   19 & 4 & 52 & 687 & 10128 & 157635 \\
  \hline
 \end{array}
\]
\caption{First values of $\{b_n^s\}$ for $1\leq s\leq 11$.}
\label{tablaa}
\end{table}

According to Conjecture \ref{suff_conj}, we expect a prime $p$ to appear at least $k$ times before getting to $a_{p^{k+1}}$. Since on every step we are multiplying by $(1+b_n)$, in fact primes appear on $a_n$ even more times than expected, specially the smallest. When $n$ grows (Table \ref{tablak}), it seems like we have much more factors than we need. Anyway, a result like the one stated on the Conjecture is needed if we want to prove that there are no composite numbers on the sequence. As we already saw, there is a very strong numeric evidence in order to suspect that it is true. By looking at Table \ref{tablaa}, in fact it is clear that the case $k=2$ is the sharpest one. And for it, the number of primes congruent with $p-1$ and less then $p^2$ grows with $p$ (see Figure \ref{RR}).

\

\paragraph*{Acknowledgments:}The author thanks Benoit Cloitre for pointing out the reference \cite{markus} and an anonymous referee for his/her constructive comments.


\end{document}